\documentclass[12pt]{amsart}

\usepackage[utf8]{inputenc}

\usepackage[a4paper,nomarginpar]{geometry}
\usepackage[english]{babel}
\usepackage{enumitem}
\usepackage{amsmath,amsfonts,amssymb,amsthm}
\usepackage{amscd}
\usepackage{tikz}
\usepackage{bigints}
\usepackage{xcolor}
\usepackage{framed}
\usepackage{centernot}

\allowdisplaybreaks
\makeatletter
\def\subsection{\@startsection{subsection}{2}%
	\z@{.5\linespacing\@plus.7\linespacing}{.3\linespacing}%
	{\normalfont\bfseries}}
\makeatother

\newtheorem{theorem}{Theorem}[section]
\newtheorem{lemma}[theorem]{Lemma}

\newtheorem{corollary}[theorem]{Corollary}
\newtheorem{prob}[theorem]{Problem}

\newtheorem{definition}[theorem]{Definition}
\newtheorem{example}[theorem]{Example}

\newcommand{\Z}{\mathbb{Z}}

\newcommand{\R}{\mathbb{R}}

\newcommand{\cB}{\mathcal{B}}

\newcommand{\cH}{\mathcal {H}}

\DeclareRobustCommand{\rchi}{{\mathpalette\irchi\relax}}
\newcommand{\irchi}[2]{\raisebox{\depth}{$#1\chi$}}

\begin{document}

\title[Chaos and frequent hypercyclicity for Composition Operators]{Chaos and frequent hypercyclicity for Composition Operators}



\author{Udayan B. Darji}
\address{Department of Mathematics, University of Louisville, Louisville,
	KY 40292, USA.}
\curraddr{}
\email{ubdarj01@gmail.com}
\thanks{The first author was supported by grant {\#}2017/19360-8 S\~ao Paulo Research Foundation (FAPESP)}

\author{Benito Pires}
\address{Departamento de Computa\c c\~ao e Matem\'atica, Faculdade de Filosofia,
	Ci\^encias e Letras, Universidade de S\~ao Paulo, Ribeir\~ao Preto, SP,
	14040-901, Brazil.}
\curraddr{}
\email{e-mail:benito@usp.br}
\thanks{The second author was partially supported by grants {\#}2018/06916-0 and  {\#}2019/10269-3
 S\~ao Paulo Research Foundation (FAPESP)}

\date{}

\dedicatory{}

\begin{abstract}  The notions of chaos and frequent hypercyclicity enjoy an intimate relationship in linear dynamics. Indeed, after a series of partial results, it was shown by Bayart and Ruzsa  in 2015 that for backward weighted shifts on $\ell_p(\mathbb{Z})$, the  notions chaos and frequent hypercyclicity coincide. It is with some effort that one shows that these two notions are distinct. Bayart and Grivaux  in 2007 constructed  a non-chaotic frequently hypercyclic weighted shift on $c_0$. It was only in 2017 that Menet  settled negatively whether every chaotic operator is frequently hypercylic. In this article, we show that for a large class of composition operators on $L^p$-spaces the notions of chaos and frequent hypercyclicity coincide. Moreover, in this particular class an invertible operator is frequently hypercyclic if and only if its inverse is frequently hypercyclic. This is in contrast to a very recent result of Menet  where an invertible frequently hypercyclic operator on $\ell_1$ whose inverse is not frequently hypercyclic is  constructed.  
\end{abstract}

\maketitle

\section{Introduction}
Let $\cH$ be a separable Banach space and $T: \cH \rightarrow \cH$ a continuous linear operator.
$T$ is {\em chaotic} if $T$ has a dense orbit and the set of periodic points of $T$ is dense in $X$. Chaotic operators are well studied in linear dynamics and well understood. For example, Grosse-Erdmann in \cite{GrosseStudia2000} gave a complete characterization of backward weighted shifts which are chaotic. 

The notion of frequent hypercyclicity is a quantitative version of hypercyclicity. It was introduced by Bayart and Grivaux in \cite{BayGri2006} in the linear setting, but it also makes sense  for Polish dynamical systems. Although such a notion is purely topological, it is related to measure-theoretical features of the topological dynamical system $(\cH, T)$. More specifically by Birkhoff's Pointwise Ergodic Theorem it follows that if $T$ admits an invariant ergodic Borel probability measure with full support, then
$T$ is frequently hypercyclic. Conversely, Grivaux and Matheron showed in \cite{GM2014} that if $\mathcal{H}$ is a reflexive Banach space, then any frequently hypercyclic operator $T$ on $\mathcal{H}$ admits a continuous invariant Borel probability measure  with full support.
The notion of frequent hypercyclicity is tricky: in fact it took a span of 10 years between the pioneer partial works \cite{BayGri2006,GEP} and the full characterization of frequently hypercyclic backward weighted shifts obtained in \cite{BayRuz2015}. 

An important problem in linear dynamics is distinguishing between chaotic operators and frequently hypercyclic operators. Bayart and Grivaux constructed in \cite{BayGri2007} a frequently hypercyclic continuous linear operator that is not chaotic while Menet showed in \cite{M2017} that there exist chaotic continuous linear operators that are not frequently hypercyclic. However, for many natural classes of operators such as  backward weighted shifts on $\ell_p(\mathbb{Z})$, these two notions coincide. 

Another important problem concerning frequently hypercyclic operators is whether the inverse of every invertible frequently hypercyclic operator is frequently hypercyclic. This was solved in the negative very recently by Menet \cite{menet2019inverse}. However, again for natural classes of operators such as backward weighted shifts on $\ell_p(\mathbb{Z})$ the inverse of an invertible frequently hypercyclic operator is frequently hypercyclic. 

Our aim in this article is to give a very large class of composition operators for which the notions of chaos and frequent hypercyclicity coincide (see Theorem~\ref{thm:FreqHyP} and Corollary~\ref{cor:ergodic}). Moreover, as the inverse of an invertible chaotic operator is chaotic, we also obtain that in our class the inverse of a frequently hypercyclic operator is also frequently hypercyclic (See Corollary~\ref{cor:invertible}). 

A powerful method for constructing frequently hypercyclic operator is to apply the Frequent Hypercyclicity Criterion (FHC). This criterion was introduced by Bayart and Grivaux in \cite{BayGri2006} and strengthened by Bonilla and Grosse-Erdmann in \cite{GrosseStudia2000}. We also explore the relationship between operators satisfying (FHC) and chaotic operators.  Theorem~\ref{SCFHC} and Theorem~\ref{thm:GenSCChar} imply that  in our setting  chaotic operators satisfy (FHC). Moreover, we prove the partial converse for composition operators on $L^p(X)$, $p \ge 2$.

We specifically study composition operators on $L^p(X)$ where $(X,\mathcal{B},\mu)$ is a $\sigma$-finite measure space and
$f:X\to X$ is a bijective, bimeasurable, nonsingular transformation satisfying $\mu\big(f^{-1}(B)\big)\le c\mu(B)$ for all $B\in\mathcal{B}$ and some $c>0$. The continuous linear operator $T_f:L^p(X)\to L^p(X)$ defined by $\varphi\mapsto \varphi\circ f$ is called the {\it composition operator induced by $f$}. In such a general setting, hypercyclic composition operators and topologically mixing composition operators on $L^p$-spaces were completely characterized in a joint work of the authors with Bayart in \cite{BDP}. Li-Yorke chaotic composition operators on $L^p$-spaces were completely characterized in a joint work of the authors with Bernardes Jr. in \cite{BJDP}. Recently, generalized hyperbolicity among composition operators was characterized by the first author, D'Aniello and Maiuriello in \cite{daniello2020generalized}. 

We would also like to point out that in a very different direction, Charpentier,  Grosse-Erdmann and Menet \cite{charpentier2019chaos} give conditions under which backward weighted shifts on K\"othe sequence spaces have the property that the notions of chaos and frequent hypercyclicity coincide.

The article is organized as follows: In Section~2 we give definitions and background results. In Section~3 we state the main results and their consequences. Section~4 consists of examples, Section~5 of proofs and Section~6 of open problems.

\section{Definitions and Background Results}
\subsection{Topological dynamics of linear operators}

Let $T:\mathcal{H}\to\mathcal{H}$ be a continuous linear operator acting on a separable Banach space $\mathcal{H}$.
\begin{definition}
    The operator $T$ is topologically transitive (or hypercyclic) if for any pair of non-empty open sets $U,V \subseteq \cH$, there is $k  > 0$ such that $T^k(U)\cap V\neq\emptyset$.  If, in addition, the set of periodic points of $T$ is dense in $\cH$, then $T$ is said to be chaotic.
\end{definition}
We recall that in the setting  of Banach spaces,  $T$ is topologically transitive if and only if $T$ admits a {\em hypercyclic vector} $\varphi$, i.e.,  $\varphi \in \cH$ such that the orbit $\{\varphi,T\varphi,T^2\varphi\ldots\}$ is dense in $\mathcal{H}$.
\begin{definition}
    The operator $T$ is  \textit{topologically mixing} if for any pair of non-empty open sets $U,V\subseteq \mathcal{H}$, there exists
$k_0\ge 0$ such that $T^k(U)\cap V\neq\emptyset$ for all $k\ge k_0$. 
\end{definition}
\begin{definition}
    A vector $\varphi\in \mathcal{H}$ is called \textit{frequently hypercyclic} if for each non-empty open set $U\subseteq \mathcal{H}$, the set of integers $\mathcal{N}(\varphi,U)=\left\{n\in\mathbb{N}: T^n\varphi\in U\right\}$ has positive lower density, that is, 
$$ \liminf_{N\to\infty} \frac{1}{N}\#\left\{1\le n\le N: T^n\varphi\in U\right\}>0\cdot
$$
The operator $T$ is called \textit{frequently hypercyclic} if it admits a frequently hypercyclic vector. 
\end{definition}
The following Frequent Hypercyclicity Criterion (FHC)  was provided by Bonilla and Grosse-Erdmann in \cite[Theorem 2.1]{BGE}. It is a strengthened version of the original criterion obtained by
  Bayart and Grivaux in  \cite[Theorem 2.1]{BayGri2006}. Its simplified reformulation is stated below in the context that we use.

\begin{theorem}[Frequent Hypercyclicity Criterion (FHC)]\label{FHC} Let $(\mathcal{H},\Vert\cdot\Vert)$ be a  separable Banach space and let  $T:\mathcal{H}\to\mathcal{H}$  be a continuous linear operator. Assume there exists a dense  subset $\mathcal{H}_0$ of $\mathcal{H}$ and a map $S:\mathcal{H}_0\to\mathcal{H}_0$ such that, for any $\varphi\in \mathcal{H}_0$,
\begin{itemize}
\item [$(a)$] The series $\sum_{n\ge 1}  T^n (\varphi)$ converges unconditionally;
\item [$(b)$] The series $\sum_{n\ge 1}  S^n(\varphi)$  converges unconditionally;
\item [$(c)$] $T S(\varphi)=\varphi$.
\end{itemize}
Then, $T$ is frequently hypercyclic, chaotic and topologically mixing.
\end{theorem}

\subsection{Measurable dynamics}

\begin{definition} A transformation $f:X\to X$ on the measure space $(X,\mathcal{B},\mu)$ is
\begin{itemize}
\item [(a)] {bimeasurable} if $f(B)\in \mathcal{B}$ and
$f^{-1}(B)\in \mathcal{B}$ for all $B\in\mathcal{B}$;
\item [(b)] {nonsingular} if $\mu\big(f^{-1}(B)\big)=0$ if 
and only if $\mu(B)=0$.
\end{itemize}
\end{definition}

\begin{definition}\label{msystem}
	A {measurable  system} is a tuple $(X,{\mathcal B},\mu, f)$, where
	\begin{enumerate}
		\item $(X,{\mathcal B},\mu)$ is a $\sigma$-finite measure space with $\mu(X)>0$; 
		\item $f : X \to X$ is a bijective bimeasurable nonsingular transformation;
		\item there is  $c > 0$ such that
		\begin{equation}\label{condition}
			\mu(f^{-1}(B)) \leq c \mu(B) \ \textrm{ for every } B \in {\mathcal B}.
			\tag{$\star$}
		\end{equation}
	\end{enumerate}
	If both $f$ and $f^{-1}$ satisfy $(*)$, then we say that the measurable system is invertible. 
\end{definition}

\begin{definition}\label{coperator} Let $p\ge 1$. The {composition operator} $T_f$ induced by
a measurable system $(X,\mathcal{B},\mu,f)$ is the map
$T_f:L^p(X)\to L^p(X)$ defined by 
$$T_f: \varphi\to\varphi\circ f.
$$
\end{definition}

It is well-known that (\ref{condition}) guarantees that $T_f$ is a continuous linear operator. We refer the reader to \cite{SM} for a detailed exposition on compositions operators.

\begin{definition}\label{2def} A measurable transformation $f:X\to X$ on the measure space $(X,\mathcal{B},\mu)$ is
\begin{itemize}
\item [(a)] {conservative} if for each measurable set
$B$ of positive $\mu$-measure, there is $n\ge 1$ such that
$\mu\big(B\cap f^{-n}(B)\big)>0$;
\item [(b)] {dissipative} if  there exists $ W\in\mathcal{B}$ such
that $f^n(W)$, $n\in\mathbb{Z}$, are pairwise disjoint and
$X=\bigcup_{n\in\mathbb{Z}} f^n(W)$. 
\end{itemize}
The measurable system $(X,\mathcal{B},\mu,f)$ is called conservative (respectively, dissipative) if $f$ is conservative (respectively, dissipative).
\end{definition}
We say that a set $A\subseteq X$ is \textit{$f$-invariant} if $f^{-1}(A)=A$. 
\begin{theorem}[Hopf \cite{AA,K}]
	Let  $(X, {\mathcal B}, \mu,f)$ be a measurable system. Then, $X$ is the union of two disjoint $f$-invariant sets ${\mathcal C}(f)$ and ${\mathcal D}(f)$, called the conservative and the dissipative parts of $f$, respectively, such that $f\vert_{\mathcal{C}(f)}$ is conservative and
	$f\vert_{\mathcal{D}(f)}$ is dissipative.
\end{theorem}


\begin{definition} Let $(X,\mathcal{B},\mu,f)$ be a measurable system. A measurable set $W\subseteq X$ is a {\em wandering set} if the sets $f^n(W)$, $n\in\mathbb{Z}$, are pairwise disjoint. The system $(X,\mathcal{B},\mu,f)$ is said to be {\it generated by a wandering set $W$} if $X=\bigcup_{n\in\mathbb{Z}}f^n(W)$.
\end{definition}
In the sequel, we let $\mathcal{B}(W)=\{B\cap W: B\in\mathcal{B}\}$. 

\begin{definition} \label{defnBD}
	We say that a dissipative system $(X,{\mathcal B},\mu, f)$ is of {\it bounded distortion} if there exist a wandering set $W$ of finite positive $\mu$-measure and $K>0$ such that
	\begin{itemize}
	    \item [(i)] $W$ generates $(X,\mathcal{B},\mu,f)$, i.e,
	    $X=\bigcup_{n\in\mathbb{Z}}f^n(W)$;
	    \item [(ii)] For all $n \in \mathbb Z$ and  $C \in {\mathcal B}(W)$ with positive $\mu$-measure,
	    \begin{equation*}
		\dfrac{1}{K} \frac{\mu(f^n(W))}{\mu(W)} \leq \frac{\mu(f^n (C))}{\mu (C)} \leq K \frac{\mu(f^n(W))}{\mu(W)}.
	\end{equation*}.
	
	\end{itemize}
\end{definition}

Notice that in the definition of dissipative system (see Definition \ref{2def}), we do not require that $W$ have finite measure.


\section{Statement of the Main Results}

{In this section we introduce a condition called {\it Summability Condition} or simply {\it Condition (SC)} that is useful for constructing composition operators that are simultaneously
chaotic, topologically mixing and frequently hypercyclic. We split our results into two subsections.
In the first subsection, we introduce Condition (SC) and we explore its relation to the Frequent Hypercyclicity Criterion (FHC). In particular, we show that if we add the hypothesis that
$f$ is dissipative, then we obtain that Condition (SC) is equivalent to $T_f$ being chaotic. In the second subsection we show that if $f$ is a dissipative transformation of bounded distortion, then Condition (SC) is equivalent to $T_f$ being frequently hypercyclic.}

\subsection{The Summability Condition  and the Frequent Hypercyclicity Criterion }

Let $(X,\mathcal{B},\mu,f)$ be a measurable system. We  say that $f$ satisfies the Summability Condition (SC) if for each $\epsilon>0$ and $B\in\mathcal{B}$ with $\mu(B)<\infty$, there exists a measurable set $ B'\subseteq B$ such that
\begin{equation}\tag{SC}
\mu\big(B{\setminus }B'\big)<\epsilon\quad\textrm{and}\quad\sum_{n\in\mathbb{Z}} \mu\big(f^n(B')\big)<\infty.
\end{equation}
The following result shows that the Summability Condition (SC)   is the natural  translation of the Frequent Hypercyclicity Criterion (FHC) to the composition operator framework. 
\begin{theorem}\label{SCFHC} Let $(X,\mathcal{B},\mu,f)$ be a measurable system. For all $p\ge 1$,  (SC) implies  (FHC). Moreover, (SC) and (FHC) are equivalent for all  $p\ge 2$. 
\end{theorem}
{Now we will provide a list of results that show how Condition (SC) is useful to characterize when a transformation $f$ is dissipative and when the composition operator $T_f$ is frequently hypercyclic or chaotic. The results below are true for all $p\ge 1$.} 

\begin{theorem}[(SC) Characterization, General Case]\label{thm:GenSCChar} Let $(X, \cB, \mu, f)$ be a measurable system and $T_f: L^p(X) \rightarrow L^p(X)$ be the associated composition operator. The following statements are equivalent.
\begin{enumerate}[label=(\alph*), align=left, leftmargin=*]

 	\item $f$ satisfies Condition (SC);
 	\item $f$ is dissipative and $T_f$ has a dense set of periodic points. 
 \end{enumerate} 
  Moreover, any of the above implies that  $T_f$ is chaotic, topologically mixing and frequently hypercyclic.
   \end{theorem}

\begin{theorem}[(SC) Characterization, $\mu(X) < \infty$]\label{thm:FinMeasSCChar}
	Let $(X, \cB, \mu, f)$ be a measurable system with $\mu(X)<\infty$ and $T_f: L^p(X) \rightarrow L^p(X)$ be the associated composition operator. Then, the following statements are equivalent.
\begin{enumerate}[label=(\alph*), align=left, leftmargin=*]
		\item  $f$ satisfies Condition (SC);
		\item  $f$ is dissipative.
	\end{enumerate}
Moreover, any of the above implies that  $T_f$ is chaotic, topologically mixing and frequently hypercyclic.
\end{theorem}

In the sequel, we will need the following definition. Let $X$ be a metric space and $f:X\to X$ be a map. We say that $x\in X$ is \textit{recurrent} (for $f$) if for every open set $U$ containing $x$, there exists an integer $n\ge 1$ such that
$f^n(x)\in U$.  

\begin{corollary}\label{recpc} Let $(X, \cB, \mu, f)$ be a measurable system where $X$ is a metric space and $\mu$ is a Borel finite measure. Let $T_f: L^p(X) \rightarrow L^p(X)$ be the associated composition operator. If the set of recurrent points of $f$ has $\mu$-measure zero, then $f$ satisfies Condition (SC) and hence $T_f$ is chaotic, topologically mixing and frequently hypercyclic.
\end{corollary}

\begin{corollary}[(SC) Characterization,  $f$ dissipative]\label{cor:chaotic}
Let $(X, \cB, \mu, f)$ be a dissipative system and $T_f: L^p(X) \rightarrow L^p(X)$ be the associated composition operator. Then, the following statements are equivalent.
\begin{enumerate}[label=(\alph*), align=left, leftmargin=*]
  \item $f$ satisfies Condition (SC);
    \item $T_f$ is chaotic;
    \item $T_f$ has dense set of periodic points. \end{enumerate}
\end{corollary}
\begin{proof}
	That (a) implies (b) follows from Theorem~\ref{thm:GenSCChar}. That (b) implies (c) is simply the definition. That (c) implies (a) follows from Theorem~\ref{thm:GenSCChar}.
\end{proof}

\subsection{Bounded Distortion and Frequent Hypercyclicity}
\begin{theorem}[Necessary condition for frequent hypercyclicity] \label{thm:FreqHyPnec} Let $(X,\mathcal{B},\mu, f)$ be a measurable system with associated composition operator $T_f$ frequently hypercyclic. Then for every wandering set $W$ with positive finite $\mu$-measure, the following inequality holds
 $$\sum_{n\in\mathbb{Z}} \left\Vert\dfrac{\mathrm{d}\mu}{\mathrm{d}\big(\mu\circ f^n\big)}\bigg|_W \right\Vert_{\infty}^{-1}<\infty\cdot$$
\end{theorem}
\begin{theorem}[Frequent Hypercyclicity Characterization] \label{thm:FreqHyP}
Let $(X, \cB, \mu, f)$ be a dissipative system of bounded distortion and $T_f: L^p(X) \rightarrow L^p(X)$ be the associated composition operator. Then, the following statements are equivalent.
\begin{enumerate}[label=(\alph*), align=left, leftmargin=*]
    \item $f$ satisfies Condition (SC);
    \item $T_f$ is frequently hypercyclic;
    \item $T_f$ is chaotic.

\end{enumerate}
\end{theorem}
\begin{corollary}\label{cor:invertible}
Let $(X, \cB, \mu ,f )$ be an invertible dissipative system of bounded distortion and $T_f ,T_{f^{-1}} :L^p(X) \rightarrow L^p(X)$ be the associated composition operators. Then, $T_f$ is frequently hypercyclic (respectively, chaotic) if and only if $(T_f)^{-1}$ is.
\begin{proof}
This follows from the fact that $(T_f)^{-1} = T_{f^{-1}}$.
\end{proof}
\end{corollary}
We say that $f$ is \textit{ergodic} if every $f$-invariant set $A\in\mathcal{B}$ satisfies 
$\mu(A)=0$ or $\mu(X{\setminus}A)=0$.
\begin{corollary}\label{cor:ergodic} Let $(X,\mathcal{B},\mu, f)$ be a dissipative system with a purely atomic measure $\mu$. Furthermore, assume that $f$ is ergodic.  Then the following statements are equivalent.
\begin{enumerate}[label=(\alph*), align=left, leftmargin=*]
\item   $T_f$ is frequently hypercyclic;
\item  $T_f$ is chaotic;
\item  $\mu$ is finite.
\end{enumerate}
\end{corollary}
\begin{proof}
\item $(a)\implies (c)$ Assume that $T_f$ is frequently hypercyclic. Let $x\in X$ be an atom of $\mu$. Since $f$ is dissipative and ergodic, we have that $W=\{x\}$ is a  wandering set that generates $X$, thus $\mu(X)=\sum_{n\in\mathbb{Z}}\mu\big(f^n(W)\big)$. Applying Theorem \ref{thm:FreqHyPnec} and using the fact that $\textrm{d}\mu/\textrm{d}\big(\mu\circ f^n\big)$ on $W$ is  equal to the constant $\mu(W)/\mu\big(f^n(W)\big)$, we obtain that 
\[\frac{\mu(X)}{\mu(W)}=\sum_{n\in\mathbb{Z}}\frac{ \mu\big(f^n(W)\big)}{\mu(W)}=\sum_{n\in\mathbb{Z}} \left\Vert\dfrac{\mathrm{d}\mu}{\mathrm{d}\big(\mu\circ f^n\big)}\bigg|_W \right\Vert_{\infty}^{-1}<\infty,\]
implying that $\mu$ is finite. 
\item  $(c)\implies (b)$ It follows from Theorem \ref{thm:FinMeasSCChar} that $f$ satisfies Condition (SC) and hence it is chaotic.  

\item $(b)\implies (a)$  As $f$ is dissipative and $T_f$ chaotic, by Corollary \ref{cor:chaotic}, we have that $f$  satisfies Condition (SC) and hence it is frequently hypercyclic. 
\end{proof}

\section{Applications and Examples}
In this section we give some applications of our main theorems. We also give some examples which show that our theorems are sharp. 

The first result is an application of Corollary \ref{recpc}. It shows that a large class of natural, simple maps $f$ on $\R^d$ yields complex behavior of $T_f$. We recall that a linear isomorphism $L:\mathbb{R}^d\to\mathbb{R}^d$ is \textit{hyperbolic} if $L$ has no eigenvalue of modulus $1$.  
  \begin{theorem}\label{thm:App} Consider a measurable system $(\R ^d, \cB,\mu, f)$ where $\mu$ is a Borel measure on $\mathbb{R}^d$,
  $\mu(\R ^d) <\infty$, $\mu(\{0\})=0$ and $f$ is a hyperbolic linear isomorphism. Then, $T_f$ is chaotic, topologically mixing and frequently hypercyclic.
\end{theorem} 
\begin{proof} By \cite[Propositions 2.9, 2.10]{JPWM}, there exist $f$-invariant subspaces $E^s$ and
$E^u$ of $\mathbb{R}^d$  with $\R ^d = E^s \bigoplus E^u$ and an adapted norm $\Vert\cdot\Vert_1$ on $\mathbb{R}^d$ with respect to which the map
$f_s=f\vert_{E^s}$ is a contraction and the map $f_u=f\vert_{E^u}$ is a dilation (thus its inverse is a contraction). In this way,
if $x\in\mathbb{R}^d$ then either $\lim_{n\to \infty} { \| f^n(x) \|}=0$ or $\lim_{n\to\infty} \Vert f^n(x)\Vert=\infty$ (with respect to any norm $\Vert \cdot\Vert$). Hence, the set of recurrent points of $f$ equals $\{0\}$, which has $\mu$-measure $0$ by hypothesis. Hence, by Corollary \ref{recpc}, the proof follows. 
\end{proof}

The following is a concrete application of Corollary \ref{recpc}.
\begin{example}\label{ex:SC}
 Let $f:\mathbb{R}\to\mathbb{R}$ be a non-identity affine map of the form $x\mapsto ax+b$, $ 0 < |a| \le 1$. Let $\mu$ be the probability measure on $\mathbb{R}$ defined by $\mu(J)=\frac12\int_J e^{-\vert t\vert } \mathrm{d}t$ for all interval $J\subseteq \mathbb{R}$. Then,  $T_f$ is chaotic, topologically mixing and frequently hypercyclic.
\end{example}
\begin{proof}
We first show that $(X,\mathcal{B},\mu,f)$ is a measurable
system, i.e., that 
Condition~(\ref{condition}) in Definition \ref{msystem} holds. Indeed, such is the case: 
\begin{equation*}\label{aux1}
 \frac{\mu\big(f(J)\big)}{\mu(J)}=
\frac{\frac12\int_{f(J)} e^{-\vert t\vert} \mathrm{d}t}{\frac12\int_J e^{-\vert t\vert}\mathrm{d}t}=\frac{|a|\int_J e^{-\vert a t +b\vert}\,\mathrm{d}t}{\int_J e^{-\vert t\vert}\,\mathrm{d}t} \ge \frac{|a| e^{-\vert b\vert}\int_J e^{-\vert a t \vert}\,\mathrm{d}t}{\int_J e^{-\vert t\vert}\,\mathrm{d}t} \ge |a| e^{-\vert b\vert},
\end{equation*}
where the last inequality follows from the fact $0 < \vert a \vert \le 1$ implies that $e^{-\vert a t\vert} \ge {e^{-\vert t\vert}}$ for all $t$.

The transformation $f$ has no recurrent points when $a=1$ as, in this case, $b\neq 0$ and $f^n(x)=x+ nb$ for all $x\in \mathbb{R}$ and $n\in\mathbb{N}$. For $a \neq 1$, some elementary computation shows that $f$ has exactly one recurrent point, namely the fixed point of $f$, $x = \frac{b}{1-a}$.  In either case, as $\mu$ is non-atomic, the set of recurrent points has measure zero.
By Corollary \ref{recpc}, we have that $T_f$ is chaotic, topologically mixing and frequently hypercyclic.   
\end{proof}

The next example shows that Corollary \ref{cor:chaotic} is sharp in the sense that the hypothesis of dissipativity cannot be removed. 
\begin{example}\label{Exe1} There exists a measurable system $(X, \cB, \mu,f)$ such that  $T_f: L^p(X) \rightarrow L^p(X)$ is chaotic but $f$ does not  satisfy Condition (SC).   
\end{example}
\begin{proof}
We use an example from our earlier work \cite{BDP}. In particular, our $X$ is an odometer and $f$ is the $+1$-map.

For $i \ge 1$, we let $\Z_i =\{0, \ldots, i-1 \}$ be integers modulo $i$. We let 
$X = \Pi_{i=1}^{\infty} A_i$,
where $A_i = \Z_2$ for $i$ even and $A_i = \Z_{2i}$ for $i$ odd. We put the discrete topology on $A_i$ and the associated product topology on $X$. Endowed with this topology, $X$ is homeomorphic to the Cantor space. 

We let $\cB$ be the collection of Borel subsets of $X$. We define a product measure $\mu$ on $X$ by defining  a probability measure $\mu_i$ on $A_i$ as follows: 
\begin{equation*}\label{even}
\mu_i(0)=\mu_i(1)=\dfrac12,\quad \textrm{if $i$ is even}
\end{equation*}
and
\begin{equation*}\label{muij}
\mu_i(j)=\begin{cases}
\dfrac{1-2^{-i}}{i} & \textrm{for $j\in \{0,\ldots,i-1\}$}\\[0.2in]
\phantom{aa}\dfrac{2^{-i}}{i} & \textrm{for $j\in \{i,\ldots,2i-1\}$}
\end{cases},\quad \textrm{if $i$ is odd.}
\end{equation*}
Hence, $X=(X,\mathcal{B},\mu)$ becomes a topological Borel probability space. 

The map $f:X\to X$ is simply the $+1$-map with carryover. It was shown in \cite{BDP} that $T_f$ is a well-defined continuous linear operator which is topologically transitive but not topologically mixing. We will show that $T_f$ has a dense set of periodic points, implying that $T_f$ is chaotic. As $T_f$ is not topologically mixing,  by Theorem \ref{thm:GenSCChar} we have that $f$ does satisfy Condition (SC).
  
  We recall that the open subsets of $X$ are countable unions of disjoint basic cylinders, i.e., sets of the form
$$[a_1,\ldots,a_i] := \left\{(x_1,\ldots,x_i,x_{i+1},\ldots) \in X : x_1 = a_1,...,x_i = a_i\right\}.$$
We claim that $\rchi_{[a_1,\ldots,a_i]}$ is a periodic point of $T_f$. In fact, if $N=\vert A_1\vert \cdot \vert A_2\vert \cdots \vert A_i\vert$, where $\vert A_j\vert$ denotes the cardinality of $A_j$, and $x=(x_1,x_2,\ldots)\in [a_1,\ldots,a_i]$,
then $f^{N}(x)=(x_1,x_2,\ldots,x_i,y_{i+1},\ldots)$, that is , $f^{N}(x)\in [a_1,\ldots,a_i]$. Hence, $f^{-N}\big([a_1,\ldots,a_i]\big)=[a_1,\ldots,a_i]$. In this way, $$T_f^N \rchi_{[a_1,\ldots,a_i]}=\rchi_{f^{-N}\big([a_1,\ldots,a_i]\big)}=\rchi_{[a_1,\ldots,a_i]},$$
implying that $\rchi_{[a_1,\ldots,a_i]}$ is a periodic point of $T_f$. 

Now it is easy to verify that the characteristic function of the finite union of cylinders is also a periodic point. From this one can easily show that the collection of simple functions of the form $\sum_{i=1}^m a_i \cdot \rchi_{C_i}$, where $C_i$ is the finite union of cylinders, is dense in $L^p(X)$, completing the proof. 
\end{proof}
The following example shows that there are simple situations where the full strength of Theorem~\ref{thm:GenSCChar} is realized. It also shows that the hypothesis of ergodicity is necessary in Corollary~\ref{cor:ergodic}.
\begin{example}\label{exefinal} There exists a dissipative system $(X, \cB, \mu,f)$ with a purely atomic  measure $\mu$ such that $\mu(X)=\infty$ and  $T_f$ satisfies Condition (SC). 
 \end{example}
 \begin{proof}
 Let $X=\mathbb{Z}\times \mathbb{Z}$ and $f:X\to X$ defined by $f\big((i,j)\big)=(i,j+1)$, $i,j\in\mathbb{Z}$. Let $\mathcal{B}=2^X$ be the discrete $\sigma$-algebra and $\mu:\mathcal{B}\to [0,\infty]$ be the $\sigma$-finite measure defined by
 $\mu\big(\{(i,j)\}\big)=2^{-\vert j\vert}$. Clearly, $\mu(X)=\infty$. 
 
 Now let us verify Condition (SC). Let $B\in\mathcal{B}$ and $\epsilon>0$ with $\mu(B)<\infty$. As $\mu(B)<\infty$, there is $L \ge 1$ such that
 $ \mu  \left (B {\setminus}  \left ( [-L, L ] \times [-L, L ]  \right ) \right ) < \epsilon$.  Let $B' = B \cap  ([-L, L ] \times [-L, L ]) $. Clearly, $\mu (B {\setminus} B' ) < \epsilon$. Next we observe that for any $(i, j) \in X$, we have that $\sum_{n \in \mathbb{Z}} \mu (f^n (i,j)) =3$. As $B' \subseteq [-L, L ] \times [-L, L ]$, we have that
 $\sum_{n \in \mathbb{Z}} \mu (f^n (B')) \le 3\cdot(2L+1)^2$. 
 \end{proof}

 \section{Proofs of the Main Results}
\begin{lemma}[Orlicz {\cite[Theorem 4.2.1]{Orlicz}}]\label{Orliczthm} Let $(X,\mathcal{B},\mu)$ be a $\sigma$-finite measure space. Let the series $\sum_{n\in\mathbb{N}} \varphi_n$ of elements of $L^p(X)$ converge unconditionally. Then for each $1\le p\le 2$, the series $\sum_{n\in\mathbb{N}} \Vert \varphi_n\Vert_p^2$ converges, and for each $2\le p<\infty$, the series $\sum_{n\in\mathbb{N}}\Vert \varphi_n\Vert_p^p$ converges.
\end{lemma}

\begin{lemma}\label{lem1proofs} Let $(X,\mathcal{B},\mu,f)$ be a measurable system.
If $f$ satisfies Condition (SC), then $f$ is dissipative.
\end{lemma}
\begin{proof}
By the Hopf Decomposition Theorem, we may write
$X=\mathcal{C}(f)\cup\mathcal{D}(f)$ as the union of the conservative and the dissipative parts of $f$, respectively. We will prove that
$\mu\big(\mathcal{C}(f)\big)=0$. By way of contradiction,
suppose that $B\subseteq \mathcal{C}(f)$ is a measurable set
with $0<\mu(B)<\infty$. 

We claim that there exist $W\subseteq{B}$ and $N\in\mathbb{N}$ such that $\mu(W)>0$ and $f^n(W)\cap W=\emptyset$ for all $n\in\mathbb{N}$ satisfying $\vert n\vert\ge N$. By  Condition (SC), there exist
a measurable set $B'\subseteq B$ and $N\in\mathbb{N}$ such that
$$ \mu(B{\setminus} B')<\dfrac{\mu(B)}{4}\quad\textrm{and}\quad \sum_{\vert n\vert\ge N} \mu\big(f^n(B')\big)<\dfrac{\mu(B)}{4}.
$$
Let $W=B'{\big\backslash} \bigcup_{\vert n\vert \ge N} f^n(B')$. Notice that $$\mu(B'{\setminus} W)\le \sum_{\vert n\vert\ge N} \mu\big(f^n(B')\big)<\frac{\mu(B)}{4}.$$
In this way, $\mu(B{\setminus} W)<\mu(B)/2$, which yields
$\mu(W)>0$.
Moreover, $f^n(W)\cap W=\emptyset$ for all $n\in\mathbb{Z}$ satisfying
$\vert n\vert \ge N$. This proves the claim.

{Now let $A=\{x\in W: f^n(x)\not\in W,\forall n\ge 1\}$. Since $A\subseteq \mathcal{C}(f)$, we have that $\mu(A)=0$. Hence, the set $W'=W{\setminus}\bigcup_{n\in\mathbb{Z}}f^n(A)$ has positive $\mu$-measure. Moreover, if $x\in W'$, then $f^n(x)\in W$ for infinitely many positive $n$, contradicting the claim.} Therefore,
$\mu\big(C(f)\big)=0$ and $f$ is dissipative.
\end{proof}

\subsection{Proof of Theorem \ref{SCFHC}} 

\noindent {\it Proof of $(SC)\implies (FHC)$.} Assume that $f$ satisfies (SC). We will show that
$T_f$ satisfies (FHC). By Lemma \ref{lem1proofs}, $f$ is dissipative. Let $W$ be a wandering set such that
$X=\bigcup_{n\in\mathbb{Z}} f^n(W)$. Let $\mathcal{H}=L^p(X)$ and $\mathcal{H}_0\subseteq\mathcal{H}$ be defined as follows: $\varphi\in\mathcal{H}_0$ if and only if there exist
$m\ge 1,
a_1,\ldots,a_m\in\mathbb{R}$, $k_1,\ldots,k_m\in\mathbb{Z}$, and pairwise disjoint measurable sets $B_1,\ldots,B_m$ such that
$\varphi=\sum_{i=1}^m a_i \rchi_{B_i}$ and, for each $1\le i\le m$, $B_i\subseteq f^{k_i}(W)$ and
 $\sum_{n\in\mathbb{Z}} \mu  \big(f^n(B_i) \big)<\infty$.

We will now show that $\mathcal{H}_0$ is dense in $\mathcal{H}$. 
Given $\epsilon>0$ and
$\psi\in\mathcal{H}$, let $\varphi=\sum_{j=1}^r a_j\rchi_{D_j}\in \mathcal{H}\backslash \{0\}$ be such that $D_1,\ldots,D_r$ are pairwise disjoint measurable sets with finite positive $\mu$-measures and
$\Vert \varphi-\psi \Vert_p<\epsilon/2$. Set $M=\max\{\vert a_j\vert: 1\le j\le r\}$, then $M>0$.
As $f$ satisfies Condition (SC) and $X=\bigcup_{n\in\mathbb{Z}}f^n(W)$, there exist an integer $N>0$
and measurable sets $C_1,\ldots,C_r$ such that, for each $1\le j\le r$,
$C_j\subseteq D_j\cap \bigcup_{\vert n\vert\le N} f^n(W)$, 
\begin{equation*}
\mu(D_j{\setminus}C_j)<\left(\dfrac{\epsilon}{2 r M}\right)^p\quad\textrm{and}\quad\sum_{n\in\mathbb{Z}} \mu\big(f^n(C_j)\big)<\infty.
\end{equation*}
Let $\varphi' =\sum_{j=1}^r a_j \rchi_{C_j}$. Since $C_j\subseteq \bigcup_{\vert n\vert\le N} f^n(W)$, we have that each $C_j$ is the  union of finitely many pairwise disjoint sets $B_i$ satisfying $B_i\subseteq f^{k_i}(W)$ for some $-N\le k_i\le N$. In this way, 
$\varphi'\in \mathcal{H}_0$. Moreover,
$$ \Vert \varphi -\varphi '\Vert_p=\left\Vert \sum_{j=1}^r a_j \big(\rchi_{D_j}-\rchi_{C_j}\big)\right\Vert_p=
\left\Vert \sum_{j=1}^r a_j \rchi_{D_j{\setminus C_j}}\right\Vert_p\le M\sum_{j=1}^r \left[\mu(D_j{\setminus} C_j)\right]^{\frac1p}<\dfrac{\epsilon}{2}.
$$
In this way, $\Vert \varphi' -\psi\Vert_p<\epsilon$, completing the proof of the denseness of $\cH_0$. 

We next show that $\sum_{n\ge 1} T_f^n\varphi$  is unconditionally convergent for all $\varphi=\sum_{i=1}^m a_i \rchi_{B_i} \in \cH_0$. Since $B_i\subseteq f^{k_i}(W)$, we have that the sets
$f^n(B_i)$, $n\in\mathbb{Z}$, are pairwise disjoint. Hence, for each sequence of integers $1\le n_1<n_2<\ldots$, we have that

\begin{equation}\label{5430}
\sum_{j\ge 1} T_f^{n_j}\varphi=\sum_{i=1}^m a_i \sum_{j\ge 1}  \rchi_{f^{-n_j}(B_i)}=\sum_{i=1}^m a_i \rchi_{\bigcup_{j\ge 1}f^{-n_j}(B_i)}.
\end{equation}

Since $\sum_{n\in\mathbb{Z}}\mu\big(f^n(B_i)\big)<\infty$ for all
$1\le i\le m$, we have that $\rchi_{\bigcup_{j\ge 1}f^{-n_j}(B_i)}\in L^p(X)$. Thus
$\sum_{j\ge 1}T_f^{n_j}\varphi$ converges for each sequence of integers $1\le n_1<n_2<\ldots$. Therefore,  $\sum_{n\ge 1} T_f^n\varphi$  is unconditionally convergent and Condition (a) in Theorem \ref{FHC} is true.

Given $\varphi=\sum_{i=1}^m a_i \rchi_{B_i}\in \mathcal{H}_0$, let $S\varphi=\sum_{i=1}^m a_i \rchi_{f(B_i)}.$ Since $f$ is bijective,  bimeasurable and non-singular, we have that  $S\varphi\in \mathcal{H}_0$ and  $S:\varphi\to S\varphi$ is a self-map on $\mathcal{H}_0$. Moreover,  $TS\varphi=\sum_{i=1}^m a_i \rchi_{f^{-1}(f(B_i))}=\varphi$, showing that Condition (c) in Theorem \ref{FHC} is true. By proceeding as in (\ref{5430}), we can show that
$\sum_{n\ge 1}S^n\varphi$ is unconditionally convergent for all $\varphi\in\mathcal{H}_0$. We have proved that $T_f$ satisfies (FHC). \\

\noindent\textit{Proof of $(FHC)\implies (SC)$ when $p\ge 2$.} Assume the hypothesis, i.e., $T_f$ satisfies the Frequent Hypercyclicity Criterion (FHC). Let $\cH_0 \subseteq L^p(X)$ be as in the statement of (FHC). Let $0<\epsilon<\frac12$ and $B\in \mathcal{B}$ with $\mu(B)<\infty$. By the denseness of $\cH_0$, there is $\varphi\in\mathcal{H}_0$ such that \[\left\Vert \varphi-\rchi_B \right\Vert_p< {\epsilon^{\left(1+\frac1p\right)}}\quad.\]
 Let 
\[
\quad B'=\left\{x\in B: \vert\varphi(x)-1\vert \le{\epsilon}\right\}.\]
Then, 
$$  {\epsilon}\left[\mu(B{\setminus}B')\right]^{\frac1p}\le \left(\int_{B{\setminus B'}}  \vert \varphi-1\vert^p\,\textrm{d}\mu\right)^{\frac1p}=\left(\int_{B{\setminus B'}}  \vert \varphi-\rchi_B\vert^p\,\textrm{d}\mu\right)^{\frac1p}\le \Vert \varphi-\rchi_B\Vert_p< {\epsilon^{\left(1+\frac1p\right)}},
$$
showing that $\mu(B{\setminus}B')<\epsilon$.

 We next show that $\sum_{n \ge 1} \mu (f^{-n}(B')) < \infty$.
Proceeding as earlier we have that for any $n\ge 1$,
\[(1-{\epsilon})\left[\mu\big(f^{-n}(B')\big)\right]^{\frac1p} \le \left(\int_{f^{-n}(B')} \left\vert \varphi\circ f^n\right\vert^p \textrm{d}\mu\right)^{\frac1p}
\]
implying that 
\[\sum_{n\ge 1} \mu\big(f^{-n}(B')\big)< 2^p\sum_{n\ge 1}\left(\int_{f^{-n}(B')} \left\vert \varphi\circ f^n\right\vert^p \textrm{d}\mu\right) \le 2^p\sum_{n\ge 1} \Vert T_f^n\varphi\Vert_p^p.\]
 Since $\sum_{n\ge 1} T_f^n\varphi$ converges unconditionally and $p\ge 2$, by Theorem \ref{Orliczthm} we have that
  $$\sum_{n\ge 1} \mu\big(f^{-n}(B')\big)<\infty.$$

 We next show that $\sum_{n \ge 1} \mu (f^{n}(B')) < \infty$. 
 As $\{\varphi, S\varphi, \ldots, S^{n}\varphi\}\subset\mathcal{H}_0$ and $T_fS$ is the identity map on $\cH_0$, we have that $T_f^nS^n(\varphi) = \varphi$ for all $n \ge 1$.
Since $f$ is bijective, bimeasurable and non-singular, we have that 
 \[ \left [T_f^nS^n(\varphi)|_B = \varphi|_B \right ] \Longrightarrow  \left [S^n(\varphi)\circ f^n| _B  = \varphi|_B \right ] \Longrightarrow \left [  S^n (\varphi)|_{f^n(B)}  = \varphi \circ f^{-n}|_{f^n(B)} \right ].
 \]
Using this we have that 
\[(1-{\epsilon})\left[\mu\big(f^{n}(B')\big)\right]^{\frac1p} \le\left(\int_{f^{n}(B')} \left\vert \varphi\circ f^{-n}\right\vert^p \textrm{d}\mu\right)^{\frac1p}=\left(\int_{f^n(B')}\left\vert S^n (\varphi)\right\vert^p \textrm{d}\mu\right)^{\frac1p} \le \left\Vert S^n\varphi\right\Vert_p.\]
Since $\sum_{n\ge 1} S_f^n\varphi$ converges unconditionally and $p\ge 2$, by Lemma \ref{Orliczthm} 
we have that 
\begin{equation*}
\sum_{n\ge 1} \mu\big(f^{n}(B')\big)\le 2^p\sum_{n\ge 1}\left\Vert S^n(\varphi)\right\Vert_p^p<\infty,
\end{equation*} which completes the proof.

\subsection{Proof of Theorem \ref{thm:GenSCChar}}

\noindent\noindent {\it Proof of} $(a) \implies (b)$. Suppose that $f$ satisfies Condition (SC). Then, by Lemma \ref{lem1proofs}, $f$ is dissipative. By Theorem \ref{SCFHC}, $T_f$ satisfies (FHC). Hence, by Theorem \ref{FHC}, 
$T_f$ is frequently hypercyclic, topologically mixing and chaotic. In particular, $T_f$ has a dense set of periodic points. \\

\noindent{\em Proof of $(b) \Longrightarrow (a)$.} Since $f$ is dissipative, there exist a wandering set
$W\in\mathcal{B}$ of positive $\mu$-measure such that
 $X=\bigcup_{n\in\mathbb{Z}} f^n(W)$. Let $\epsilon>0$ and $B\in\mathcal{B}$ with $\mu(B)<\infty$. Let $n_0\in\mathbb{N}$ be such that if
 $$B_0=B\cap\bigcup_{\vert k\vert\le n_0}f^k(W),\quad\textrm{then}\quad \mu(B{\setminus} B_0)<\dfrac{\epsilon}{2}.$$
 By hypothesis, there exists a periodic point $\varphi\in L^p(X)$ of $T_f$ such that 
 $$ \Vert \varphi-\rchi_{B_0}\Vert_p^p\le \dfrac{1}{4^p}\cdot \dfrac{\epsilon}{2}.
 $$
 For $B'=\left\{x\in B_0: |\varphi(x)|>\dfrac{3}{4}\right\}$, we have that $\mu(B_0{\setminus} B')\le\dfrac{\epsilon}{2}$ because
 $$ \dfrac{1}{4^p} \mu(B_0{\setminus}B')\le       \int_{B_0{\setminus}B'} \left| \varphi-1\right|^p \textrm{d}\mu\le \Vert \varphi-\rchi_{B_0}\Vert_p^p\le \dfrac{1}{4^p}\cdot \dfrac{\epsilon}{2}, \quad \textrm{implying}\quad \mu\big(B_0{\setminus}B'\big)\le \frac{\epsilon}{2}.
 $$ 
 In this way, 
 $$\mu(B{\setminus}B')= \mu(B{\setminus}B_0)+\mu(B_0{\setminus}B')< \epsilon.$$
 Let $N\ge 2n_0+1$ be such that $T_f^N \varphi=\varphi\circ f^N=\varphi$ $\mu$-a.e. Hence, since $f$ is bijective,
 we have that $\varphi\circ f^{kN}=\varphi$ $\mu$-a.e. for all $k\in\mathbb{Z}$. Because $B'\subseteq B_0$ and $N\ge 2n_0+1$, we have that the sets in the family $\left\{ f^{kN}(B'):k\in\mathbb{Z}\right\}$ are pairwise disjoint. Moreover, 
 \begin{eqnarray*}
 \left(\dfrac{3}{4}\right)^p\sum_{k\in\mathbb{Z}} \mu\big(f^{kN}(B')\big) &\le &
  \sum_{k\in\mathbb{Z}}\int_{f^{kN}(B')} \vert \varphi\circ f^{-kN}\vert^p \textrm{d}\mu=
 \sum_{k\in\mathbb{Z}}\int_{f^{kN}(B')} \vert \varphi\vert^p \textrm{d}\mu =\\&=&\int_{\bigcup_{k\in\mathbb{Z}} f^{kN}(B')} \vert \varphi\vert^p \textrm{d}\mu\le \int |\varphi|^p \textrm{d}\mu,
 \end{eqnarray*}
 showing that $\sum_{k\in\mathbb{Z}} \mu\big(f^{kN}(B')\big)<\infty$. 
By (\ref{condition}), we have that
 $$ \sum_{n\in\mathbb{Z}}\mu(f^n(B'))=\sum_{i=0}^{N-1}\sum_{k\in\mathbb{Z}}\mu(f^{kN-i}(B'))\le N \max\, \{1,c,\ldots,c^{(N-1)}\}\cdot \sum_{k\in\mathbb{Z}}\mu\big(f^{kN}(B')\big)<\infty.
 $$ 
 This proves that Condition (SC) is true, which concludes the proof of $(b)\implies (a)$.

\subsection{Proof of Theorem \ref{thm:FinMeasSCChar}}

\noindent{\em Proof of $(a)\implies (b)$}. This is Lemma \ref{lem1proofs}.

\noindent{\em Proof of $(b)\implies (a)$}. Assume $f$ is dissipative and $\mu(X)<\infty$. Then,
there exists a wandering set $W$ such that $X=\bigcup_{n\in\mathbb{Z}}f^n(W)$. Let $\epsilon>0$ and $B\in\mathcal{B}$  with $\mu(B) < \infty$ be given. Let us verify Condition (SC). Let $N\in\mathbb{N}$ be such that
$\mu\left(B{\setminus} \bigcup_{\vert n\vert\le N} f^n(W)\right)<\epsilon$. Then $B'=B\cap \bigcup_{\vert n\vert\le N} f^n(W)$ verifies Condition (SC) because
$\mu(B{\setminus}B')<\epsilon$ and by the dissipativity of $f$, we have that
$$ \sum_{n\in\mathbb{Z}}\mu\big(f^n(B')\big)\le(2N+1)\mu(X)<\infty,
$$
which concludes the proof.

\subsection{Proof of Corollary \ref{recpc} }

The proof consists in verifying Condition (SC). Hence, to begin with,
let $\epsilon>0$ and $B\in\mathcal{B}$ with $0<\mu(B)<\infty$ be given. We may assume that $B$ has no recurrent points.
 For each $n\ge 1$, set
$$ A_n=\left\{x \in B: \{f(x), f^2(x),\ldots\}\cap B\left(x,\frac1n\right)=\emptyset\right\},
$$
where $B\left(x,\frac1n\right)=\{y\in X: d(y,x)<\frac1n\}$, where $d$ is the metric on $X$. Notice that $A_1\subseteq A_2\subseteq \ldots$
and $B=\bigcup_{n=1}^\infty A_n$. In this way, there exist $N$ large enough so that $\mu(B{\setminus} A_N)<\frac{\epsilon}{4}$. Let $K$ be a compact subset of $A_N$ such that $\mu(A_N{\setminus}K)<\frac{\epsilon}{4}$. Let $U_1, U_2,\ldots, U_r$ be a finite collection of balls of radius $\frac{1}{2N}$ such that $K\subseteq \bigcup_{i=1}^r U_i$. Since $U_i\cap K\subseteq A_N$, we have that if $x\in U_i\cap K$, then $\{f(x),f^2(x),\ldots\}\cap U_i=\emptyset$. In particular, if $K_i=U_i\cap K$, then the following statements are true: \begin{itemize}
 \item [(a)] $K=\bigcup_{i=1}^r K_i$;
 \item [(b)] $f^n(K_i)\cap K_i=\emptyset$ for all $n\ge 1$.
 \end{itemize}
The bijectivity of $f$ and (b) imply that $W$ is a wandering set. Hence,
$$\sum_{n\in\mathbb{Z}} \mu\big( f^n(K_i)\big)\le  \mu(X).
$$
In this way,
$$ \mu(B{\setminus}K) <\epsilon\quad\textrm{and}\quad
\sum_{n\in\mathbb{Z}} \mu\big(f^n(K)\big)\le r\mu(X)<\infty.
$$
Thus, Condition (SC) is true, which concludes the proof.

\subsection{Proof of Theorem \ref{thm:FreqHyPnec}}

Throughout this section, $(X,\mathcal{B},\mu,f)$ will denote a measurable system. Given a measurable set $W\subseteq X$ with finite positive $\mu$-measure, let
 $\big(d_n(W)\big)_{n\in\mathbb{Z}}$  be the sequence defined by
\begin{equation}\label{alfa}
d_n(W)=
\left\Vert\dfrac{\textrm{d}\mu}{\textrm{d}\big(\mu\circ f^n\big)}\bigg|_W \right\Vert_{\infty}^{-1},
\end{equation}
where we set $1/\infty=0$ so that $d_n(W)$ is well-defined even in the case in which $\left\Vert \frac{\mathrm{d}\mu}{\mathrm{d}
\big(\mu\circ f^n\big)}\bigg|_W\right\Vert_{\infty}=\infty$.

\begin{lemma}\label{dnc} Let $c>0$ be as in ($\star$) and let
 $W$ be a measurable set with finite positive $\mu$-measure. Then,  $d_{n+1}(W)\ge c^{-1} d_n(W)$ for all $n\in\mathbb{Z}$.
\end{lemma}
\begin{proof} By ($\star$), $\mu\big(f^n(B)\big)\le c\mu\big( f^{n+1}(B)\big)$ for all $B\in\mathcal{B}$ and $n\in\mathbb{Z}$. Hence, the Radon-Nikodym derivatives $\textrm{d}\big(\mu\circ f^n\big)/\textrm{d}\big(\mu\circ f^{n+1}\big)$, $n\in\mathbb{Z}$, are  bounded by $c$ at $\mu\circ f^{n+1}$-almost every point, and therefore, $\mu$-a.e. In this way, for all $n\in\mathbb{Z}$,
\begin{equation}\label{auxiliary1}
\dfrac{\textrm{d}\mu}{\textrm{d}\big(\mu\circ f^{n+1}\big)}=\dfrac{\textrm{d}\mu}{\textrm{d}\big(\mu\circ f^n\big)}
\cdot\dfrac{\textrm{d}\big(\mu\circ f^{n}\big)}{\textrm{d}\big(\mu\circ f^{n+1}\big)}\le c \dfrac{\textrm{d}\mu}{\textrm{d}\big(\mu\circ f^n\big)} \quad \mu\textrm{-a.e.}
\end{equation}
Therefore, for all $n\in\mathbb{Z}$, we have that
\begin{equation}\label{er56}
 \dfrac{1}{d_{n+1}(W)}=\left\Vert \dfrac{\textrm{d}\mu}{\textrm{d}\big(\mu\circ f^{n+1}\big)}\bigg|_W\right\Vert_{\infty}\le c \left\Vert \dfrac{\textrm{d}\mu}{\textrm{d}\big(\mu\circ f^{n}\big)}\bigg|_W\right\Vert_{\infty}=c \dfrac{1}{d_n(W)},
\end{equation}
showing that $d_{n+1}(W)\ge c^{-1} d_n(W)$ for all $n\in\mathbb{Z}$.

\end{proof}

 \begin{lemma}\label{<17} Suppose that $T_f$ is frequently hypercyclic. 
  If $W$ is a wandering set with finite positive $\mu$-measure, then there exists a set $A\subset\mathbb{N}$ with positive lower density such that for each $n\in A$,
$$\sum_{m\in A}  d_{m-n}(W)<2.$$

 \end{lemma}
\begin{proof} Let $\varphi\in L^p(X)$ be a frequently hypercyclic vector for $T_f$ and let $W$ be a  wandering set with finite positive $\mu$-measure. Let
\begin{equation}\label{setA}
 A=\Big\{ n\in\mathbb{N}: \left\Vert T_f^n\varphi-\rchi_W\right\Vert_p^p<\frac{1}{2^p} \mu(W)\Big\},
\end{equation}
then $A$ has positive lower density.

By combining the definition of $A$ with the fact that $W$ is a  wandering set, we obtain for each $n\in A$,
\begin{equation}\label{W3a}
\int_W \left\vert \varphi\circ f^n-1\right\vert^p \textrm{d}\mu + \sum_{m\neq 0}\int_{f^m(W)} \left\vert \varphi\circ f^n\right\vert^p \textrm{d}\mu \le \left\Vert T_f^n \varphi-\rchi_W\right\Vert_p^p <\dfrac{1}{2^p}\mu(W).
\end{equation}
In particular, the second term of (\ref{W3a}) satisfies the following inequality for each $n\in A$,
\begin{equation}\label{eqc2}
 \sum_{m\neq 0}\int_{f^m(W)} \left\vert \varphi\circ f^n\right\vert^p \textrm{d}\mu <\dfrac{1}{2^p}\mu(W).
\end{equation}
Applying the triangle inequality to the first term of (\ref{W3a}) yields for each $n\in A$, 
\begin{equation}\label{xesp}
\left[\mu(W)\right]^{\frac1p}- \left(\int_W \left\vert\varphi\circ f^n\right\vert^p\textrm{d}\mu\right)^{\frac1p}\le \left(\int _W \left\vert\varphi\circ f^n-1\right\vert^p\textrm{d}\mu\right)^{\frac1p}<\frac12\left[\mu(W)\right]^{\frac{1}{p}}.
\end{equation}
Therefore, renaming $n$ by $k$ in (\ref{xesp}) yields
\begin{equation}\label{1839}\int_W \left\vert\varphi\circ f^k\right\vert^p\textrm{d}\mu> \frac{1}{2^p}\mu(W),\quad \forall k\in A.
\end{equation}

In what follows, keep $n\in A$ fixed and let $m\in A-n$ be arbitrary. Setting $k=n+m$ in
 (\ref{1839}) yields $\int_W \left\vert\varphi\circ f^{n+m}\right\vert^p\textrm{d}\mu> \frac{1}{2^p}\mu(W)$. By combining this with (\ref{eqc2}), we arrive at
\begin{equation}\label{eqc3}
\sum_{\substack{ m\in A-n\\ m\neq 0}}  \dfrac{\bigintsss_{f^m(W)} \left\vert \varphi\circ f^{n}\right\vert^p \textrm{d}\mu}{\bigintsss_{W}\left\vert\varphi\circ f^{n+m}\right\vert^p\textrm{d}\mu}<1.
\end{equation}
By the Change of Variables Formula, for all $m\in A-n$,
\begin{equation}\label{a49}
\int_{f^m(W)}\left\vert\varphi\circ f^n\right\vert^p\textrm{d}\mu=\int_{f^m(W)}\left\vert\varphi\circ f^{n+m}\circ f^{-m}\right\vert^p\textrm{d}\mu=\int_W \left\vert\varphi\circ f^{n+m}\right\vert^p d\big(\mu\circ f^m\big).
\end{equation}
Replacing (\ref{a49}) in (\ref{eqc3}) yields
\begin{equation}\label{arbol}
\sum_{\substack{ m\in A-n\\ m\neq  0}}  \dfrac{\bigintsss_W \left\vert\varphi\circ f^{n+m}\right\vert^p d\big(\mu\circ f^m\big)}{\bigints_{W}\left\vert\varphi\circ f^{n+m}\right\vert^p \dfrac{\textrm{d}\mu}{\textrm{d}\big(\mu\circ f^{m}\big)}\textrm{d}\big(\mu\circ f^m\big)}<1.
\end{equation}
For all $m\in A-n$,
\begin{equation}\label{999}
\int_{W}\left\vert\varphi\circ f^{n+m}\right\vert^p \dfrac{\textrm{d}\mu}{\textrm{d}\big(\mu\circ f^{m}\big)}\textrm{d}\big(\mu\circ f^m\big)\le\left\Vert \dfrac{\textrm{d}\mu}{\textrm{d}\big(\mu\circ f^m\big)}\bigg|_W \right\Vert_\infty\cdot
 \int_W \left\vert\varphi\circ f^{n+m}\right\vert^p\textrm{d}\big(\mu\circ f^m\big).
\end{equation}
Replacing  (\ref{999}) in (\ref{arbol}) and cancelling out the nonzero term $ \int_W \left\vert\varphi\circ f^{n+m}\right\vert^p\textrm{d}\big(\mu\circ f^m\big)$ yields
$$\sum_{\substack{ m\in A-n\\ m\neq 0}} \left\Vert \dfrac{\textrm{d}\mu}{\textrm{d}\big(\mu\circ f^m\big)}\bigg|_W\right\Vert_\infty^{-1}<1.
$$
Therefore, after renaming variables, we obtain
\begin{equation*}\label{eqc5}
\sum_{m\in A}  d_{m-n}(W)=
\sum_{m\in A} \left\Vert \dfrac{\textrm{d}\mu}{\textrm{d}\big(\mu\circ f^{m-n}\big)}\bigg|_W\right\Vert_\infty^{-1}=1+\sum_{\substack{ m\in A-n\\ m\neq 0}} \left\Vert \dfrac{\textrm{d}\mu}{\textrm{d}\big(\mu\circ f^m\big)}\bigg|_W\right\Vert_\infty^{-1}
<2.
\end{equation*}

\end{proof}

\begin{lemma}[{Bayart-Ruzsa \cite[Corollary 9]{BayRuz2015}}]\label{BR} Let $(\alpha_n)_{n\in\mathbb{Z}}$ be a sequence of non-negative real numbers such there exists $C>0$ such that either $\alpha_{n+1}\ge C \alpha_{n}$ for all $n\in\mathbb{Z}$, or $\alpha_{n}\ge C \alpha_{n+1}$ for all $n\in\mathbb{Z}$. Suppose that for some set $A\subset\mathbb{Z}$ with positive upper density the sequence $(\beta_n)_{n\in A}$ defined by
$ \beta_n=\sum_{m\in A} \alpha_{m-n}
$
is bounded. Then $\sum_{n\in\mathbb{Z}} \alpha_n<\infty$.
\end{lemma}

\subsection{Proof of Theorem \ref{thm:FreqHyPnec}}

 Set $\alpha_n= d_n(W)$ for all $n\in\mathbb{Z}$. By Lemma \ref{dnc}, we have that $\alpha_{n+1}\ge c^{-1} \alpha_n$ for all $n\in\mathbb{Z}$. Let $A$ be as in Lemma \ref{<17} and let $\big(\beta_n\big)_{n\in {A}}$ be defined by $\beta_n=\sum_{m\in A} \alpha_{m-n}$.
Then, by Lemma \ref{<17}, for each $n\in A$,
 $$\beta_n=\sum_{m\in A} \alpha_{m-n}= \sum_{m\in A}  d_{m-n}(W)<2.
$$
Hence, $\big(\beta_n\big)_{n\in {A}}$ is bounded. By Lemma \ref{BR}, 
$$ \sum_{n\in\mathbb{Z}} d_n(W)=\sum_{n\in\mathbb{Z}} \alpha_n<\infty.
$$
\subsection{Proof of Theorem \ref{thm:FreqHyP}}

\noindent{Proof of $(a)\implies (b)$}. This follows immediately from Theorem \ref{thm:GenSCChar}. 

\noindent Proof of $(b)\implies (a)$.
 Let $B\in\mathcal{B}$ with $\mu(B)<\infty$ and
$\epsilon>0$ be given. We will verify Condition (SC).
Let $W$ be a wandering set of finite positive $\mu$-measure satisfying the bounded distortion Conditions (i) and (ii) in  Definition \ref{defnBD}.
For each $i\in\mathbb{Z}$, set $W_i=f^i(W)$. By (i), there exists
$N\in\mathbb{N}$ such that if $B'=B\cap \bigcup_{\vert i\vert\le N}f^i(W)$, then $\mu(B{\setminus B'})<\epsilon$. We claim that
there exists $K>0$ such that for each integer $-N\le i\le N$, each $D\in\mathcal{B}(W_i)$ with positive $\mu$-measure and each $n\in\mathbb{Z}$,
$$  \frac{\mu(D)}{\mu\big(f^n(D)\big)}\le K^2\frac{\mu(W_i)}{\mu\big(f^n(W_i)\big)}.
$$
In fact, since $f$ is bijective, bimeasurable and non-singular,
any set $D\in \mathcal{B}(W_i)$ has positive $\mu$-measure
if and only if $D=f^i(C)$, where $C\in\mathcal{B}(W)$ has positive $\mu$-measure. By Condition (ii) in Definition \ref{defnBD}, we have that
\begin{eqnarray*}
\frac{\mu\big(W_i\big)}{\mu\big(f^n(W_i)\big)}&=&\frac{\mu\big(f^i(W)\big)}{\mu\big(f^{n+i}(W)\big)}\\&=&\frac{\mu\big(f^i(W)\big)}{\mu(W)}\cdot\frac{\mu(W)}
{\mu\big(f^{n+i}(W)\big)}\\&\ge& \frac{1}{K^2}\frac{\mu\big(f^i(C)\big)}
{\mu(C)}\frac{\mu(C)}{\mu\big(f^{n+i}(C)\big)}\\&=&\frac{1}{K^2}
\frac{\mu(D)}{\mu\big(f^n(D)\big)}.
\end{eqnarray*}
which proves the claim. 

Hence, since $D\in\mathcal{B}(W_i)$ is an arbitrary measurable
set of positive $\mu$-measure, we obtain
$$ \frac{\mu(W_i)}{\mu\big(f^n(W_i)\big)}\ge \frac{1}{K^2}\frac{\mathrm{d}\mu}{\mathrm{d}\big(\mu\circ f^n\big)}(x),\quad
\textrm{for $\mu$-almost every $x\in W_i$}.
$$
Therefore, since $W_i$ is a wandering set and $T_f$ is frequently hypercyclic, by Theorem \ref{thm:FreqHyPnec} we arrive at
$$ \sum_{n\in\mathbb{Z}}\frac{\mu\big(f^n(W_i)\big)}{\mu(W_i)}\le K^2 \sum_{n\in\mathbb{Z}}\left\Vert\frac{\mathrm{d}\mu}{\mathrm{d}\big(\mu\circ f^n\big)}\bigg|_{W_i}\right\Vert_{\infty}^{-1}<\infty.
$$
In this way,
$$\sum_{n\in\mathbb{Z}} \mu\big(f^n(B')\big)\le
\sum_{n\in\mathbb{Z}}\sum_{\vert i\vert\le N} \mu\big(f^n(W_i)\big)=\sum_{\vert i\vert\le N}\mu(W_i)\sum_{n\in\mathbb{Z}} \frac{\mu\big(f^n(W_i)\big)}{\mu(W_i)}<\infty.$$

\noindent{Proof of $(a)\iff (c)$}. This follows immediately from Corollary \ref{cor:chaotic}.

\section{Problems and final comments}
Below we list some problems related to our work.
\begin{prob}
In Theorem~\ref{SCFHC}, we show that (FHC) implies (SC) for $ p \ge 2$. What happens for $1 \le p <2$?
\end{prob}
\begin{prob}
Is the condition of bounded distortion necessary in Theorem~\ref{thm:FreqHyP}?
\end{prob}
\begin{prob}
In Example~\ref{Exe1} is $T_f$  frequently hypercyclic? If not, this would give an alternative solution to a result of Menet \cite{M2017}.
\end{prob}
\begin{prob}
We have thoroughly studied the case when the measurable system is dissipative. What happens in the case when the measurable system is conservative? In particular, what can be said when $X$ is an odometer endowed with the product measure? Characterizations of hypercyclic and topologically mixing operators of such type were given in \cite{bongiorno2019linear}.
\end{prob}

Concerning the extension of our results to the non-bijective case, a slight variation of the proof of Theorem \ref{thm:FreqHyPnec} leads to the result stated below. We recall that a measurable set $W\subseteq X$ is a \textit{forward wandering set} if $W, f(W), f^2(W),\ldots$ are pairwise disjoint.

\begin{theorem}[Necessary condition for frequent hypercyclicity]\label{thmps} Let $(X,\mathcal{B},\mu)$ be a $\sigma$-finite measure space and $f:X\to X$ be an injective bimeasurable map satisfying $(\star)$  with associated composition operator $T_f$ frequently hypercyclic. Then for every forward wandering set $W$ with positive finite $\mu$-measure, the following inequality holds
 $$\sum_{n\in\mathbb{N}} \left\Vert\dfrac{\mathrm{d}\mu}{\mathrm{d}\big(\mu\circ f^n\big)}\bigg|_W \right\Vert_{\infty}^{-1}<\infty\cdot$$
\end{theorem}

Let $w=(w_i)_{i\in\mathbb{N}}$ be a bounded sequence of positive real numbers and $B_{w}$ be the backward weighted shift on $\ell_p(\mathbb{N})$. By considering $X=\mathbb{N}$, $\mathcal{B}=\mathcal{P}(\mathbb{N})$ (the power set of $\mathbb{N}$), $f:i\in\mathbb{N}\to i+1\in\mathbb{N}$, $W=\{0\}$ and $\mu(\{i\})=\dfrac{1}{\left(w_0\cdots w_i\right)^p}$ in Theorem \ref{thmps}, we obtain the implication $(i)\implies (iii)$ in \cite[Theorem 4]{BayRuz2015}, which leads to the known conclusion that frequently hypercyclic backward weighted shifts on $\ell_p(\mathbb{N})$ are chaotic. Likewise, by means of Corollary \ref{cor:ergodic}, we obtain the equivalence $(i)\iff (iii)$ in  \cite[Theorem 3]{BayRuz2015} and we
 arrive at the known conclusion that  a backward weighted shift on $\ell_p(\mathbb{Z})$ is frequently hypercyclic if and only if it is chaotic. 
 
 Apart from Theorem \ref{thmps}, we do not know if the other results can be generalized to the non-bijective case. Note that we use the hypothesis that $f$ is bijective in the Hopf Decomposition Theorem.

\end{document}